\documentclass[12pt,reqno]{amsart}
\usepackage[all]{xy}
\usepackage{mathrsfs} 
\usepackage{graphicx} 
\numberwithin{equation}{section}

\theoremstyle{plain} 
\newtheorem{theorem}{Theorem}[section] 
 
\newtheorem{lemma}[theorem]{Lemma} 
\newtheorem{corollary}[theorem]{Corollary} 

\theoremstyle{definition} 
\newtheorem{remark}[theorem]{Remark}

\newcommand{\p}{{\mathbb P}}

\newcommand{\z}{{\mathbb Z}} 
\newcommand{\pj}{{{\mathbb P}^1}}
\newcommand{\pii}{{{\mathbb P}^2}}
\newcommand{\piii}{{{\mathbb P}^3}}

\newcommand{\sce}{\mathscr{E}}
\newcommand{\scf}{\mathscr{F}} 
\newcommand{\scg}{\mathscr{G}}
\newcommand{\sco}{\mathscr{O}} 
\newcommand{\sch}{\mathscr{H}}
\newcommand{\sci}{\mathscr{I}} 

\newcommand{\sck}{\mathscr{K}}
\newcommand{\scl}{\mathscr{L}}

\newcommand{\fm}{{\mathfrak m}}

\newcommand{\tH}{\text{H}} 
\newcommand{\h}{\text{h}}

\newcommand{\izo}{\overset{\sim}{\rightarrow}} 
 
\newcommand{\ra}{\rightarrow} 
\newcommand{\lra}{\longrightarrow} 
\newcommand{\xra}{\xrightarrow}  
\newcommand{\vb}{\, \vert \, } 
\newcommand{\prim}{{\, \prime}} 
\newcommand{\secund}{{\prime \prime}}
\newcommand{\Ker}{\text{Ker}\, }
\newcommand{\Cok}{\text{Coker}\, }

\newcommand{\e}{\varepsilon}

\begin{document}

\title[Vector bundles and unirationality]{On the vector bundles from Chang and 
       Ran's proof of the unirationality of $\mathcal{M}_g$, $g \leq 13$}

\author[C. Anghel]{Cristian Anghel}
\address{Institute of Mathematics ``Simion Stoilow'' of the Romanian Academy, 
         P.O. Box 1--764, 
         RO--014700, Bucharest, Romania}
\email{Cristian.Anghel@imar.ro}

\author[I. Coand\u{a}]{Iustin Coand\u{a}}
\email{Iustin.Coanda@imar.ro}

\author[N. Manolache]{Nicolae Manolache}
\email{Nicolae.Manolache@imar.ro}

\subjclass[2020]{Primary:14J60; Secondary: 14H10, 14H50}

\keywords{vector bundles on projective spaces, moduli spaces of curves, 
          unirationality, space curves}


\begin{abstract}
We combine the idea of Chang and Ran [Invent. Math. 76 (1984), 41--54] of 
using monads of vector bundles on the projective 3-space to prove the 
unirationality of the moduli spaces of curves of low genus with our 
classification of globally generated vector bundles with the first Chern 
class $c_1 = 5$ on the projective 3-space [arXiv:1805.11336] to get 
an alternative argument for the unirationality of the moduli spaces of curves 
of degree at most 13 (following the general guidelines of the method of Chang 
and Ran but with quite different effective details).    
\end{abstract}

\maketitle 
\tableofcontents 

\section*{Introduction} 

The problem of the unirationality of the moduli space $\mathcal{M}_g$ of 
curves of genus $g$ is a classical one$\, :$ see Verra \cite{v2} for a recent  
survey. We have to mention, however, the papers of Arbarello and 
Sernesi \cite{as} (who treated, in modern terms, the classical case 
$g \leq 10$), Sernesi \cite{se1} (who solved the case $g = 12$), 
Mori and Mukai \cite{mm} (who proved the uniruledness of $\mathcal{M}_{11}$), 
Chang and Ran \cite{cr} (who showed that $\mathcal{M}_{11}$, $\mathcal{M}_{13}$ 
(and $\mathcal{M}_{12}$) are unirational), and Verra \cite{v1} (who solved 
the case $g = 14$ and lower, using results of Mukai \cite{mu}). 
The subject came to our attention incidentally, as a consequence of our 
study \cite{acm1}, \cite{acm4} of globally generated vector bundles with 
small $c_1$ on projective spaces. The connection with the above mentioned 
problem stemmed from the circumstance that Chang and Ran \cite{cr} use a 
certain class of vector bundles with $c_1 = 5$ on $\piii$ to prove the 
unirationality of $\mathcal{M}_g$, for $g \leq 13$.       
More precisely, Chang and Ran \cite{cr} show that, for $g \leq 13$, there   
exists a family of nonsingular space curves of genus $g$ and degree 
$d = \lceil(3g+12)/4\rceil$ having general moduli and other ``good 
properties''. These curves, in case they exist, can be represented as the 
dependency locus of $r - 1$ global sections of certain vector 
bundles of rank $r = g - d + 4$ and, in turn, these vector bundles are the 
cohomology sheaves of some linear monads with terms depending only on $g$ and 
$d$. Finally, the good properties of the curves imply, quickly, that the space 
of these monads is rational. At this point, the vector bundles leave the stage 
and Chang and Ran concentrate on the existence of the above mentioned families  
of curves, which is the key and difficult point of their approach. In order to 
achieve this goal, the two authors use a method of Sernesi \cite{se2} which 
consists in successively attaching 4-secant conics to a curve of lower degree 
and genus, then showing that the resulting reducible curve has the necessary 
properties and, finally, smoothing that curve.       

The alternative approach we propose in this paper is to concentrate on the 
vector bundles instead of the curves.   
More precisely, we show that the cohomology sheaf of a general 
monad of the above type is a vector bundle $E$ with a number of good 
properties and that the dependency locus of $r - 1$ general global sections 
of $E$ is a nonsingular curve which has the properties required by the 
approach of Chang and Ran. The difficult point becomes, this time, the 
proof of the fact that $r - 1$ general global sections of $E$ are linearly  
dependent along a \emph{nonsingular} curve. We verify this by showing that $E$ 
is globally generated if $g \leq 12$, while for $g = 13$, where this fact is 
no longer true, we use a criterion of Martin-Deschamps and Perrin 
\cite{mdp}, recalled in Lemma~\ref{L:dependencyloci} from 
Appendix~\ref{A:monads}. Actually, if one wants to stick to vector bundles, the 
method of Chang and Ran works only in the range $8 \leq g \leq 13$. The 
approach with linear monads can be, however, extended to the range 
$5 \leq g \leq 7$, replacing the vector bundles by rank 2 reflexive sheaves. 
This approach is due to Chang~\cite[Cor.~4.8.1]{c}.

The paper is organized as follows$\, :$ we explain the method of Chang and Ran 
and our alternative approach in Section 1. 
We treat, then, in the next sections, the cases $8 \leq g \leq 11$, $g = 13$, 
and $g = 12$. We say, in the final Section 5, a few words about the cases 
$5 \leq g \leq 7$. 
We gather, in Appendix~\ref{A:monads}, some facts about monads and dependency 
loci while Appendix~\ref{A:auxiliary} contains a number of other auxiliary 
results.  

\vskip2mm 

\noindent 
{\bf Note.}\quad Our emphasis in this paper is on the usefulness of vector 
bundles on projective spaces in handling various geometric problems. The study 
of these bundles was an active area of research in the 1970s and 1980s but 
fall later into oblivion. One ignores, nowadays, many of the results and 
techniques developed during that period. 

\vskip2mm 

\noindent 
{\bf Notation.}\quad (i) We denote by $\p^n$ the projective $n$-space over an 
algebraically closed field $k$ of characteristic 0 and by $S \simeq 
k[x_0, \ldots , x_n]$ its projective coordinate ring.  

(ii) If $\scf$ is a coherent sheaf on $\p^n$ and $i \geq 0$ is an 
integer, we denote by $\tH^i_\ast(\scf)$ the graded $S$-module 
$\bigoplus_{l \in \z}\tH^i(\scf(l))$ and by $\h^i(\scf)$ the dimension of 
$\tH^i(\scf)$.  

(iii) If $E$ is a vector bundle (= locally free sheaf) on a variety $X$, we 
denote its dual $\sch om_{\sco_X}(E , \sco_X)$ by $E^\vee$. 

(iv) If $Y \subset X \subset \p^n$ are closed subschemes of $\p^n$, defined 
by ideal sheaves $\sci_X \subset \sci_Y \subset \sco_{\p^n}$, we denote by 
$\sci_{Y , X}$ the ideal sheaf of $Y$ as a subscheme of $X$, that is 
$\sci_{Y , X} = \sci_Y/\sci_X$. 

(v) If $D$ is a Cartier divisor on an equidimensional projective scheme $X$, 
we denote by $\sco_X[D]$ the associated invertible $\sco_X$-module. 



\section{The Method of Chang and Ran}\label{S:method} 

Let $g$ be an integer with $8 \leq g \leq 13$ and let $Y \subset \piii$ be a 
nonsingular, connected, nondegenerate (that is, not contained in a plane) 
space curve, of genus $g$ and (some) degree $d$. Let $H_{d,g}$ denote the 
open subset of the Hilbert scheme of subschemes of $\piii$ parametrizing 
nonsingular, connected space curves of degree $d$ and genus $g$ and let 
$[Y]$ be the point of $H_{d,g}$ corresponding to $Y$. 

\vskip2mm 

\noindent 
{\bf 1.1. The map from the Hilbert scheme to the moduli space.}\quad 
Consider the exact sequences$\, :$ 
\begin{gather*} 
0 \lra \text{T}_Y \lra \text{T}_\piii \vb Y \lra \text{N}_Y \lra 0\, ,\\  
0 \lra \sco_Y \lra \tH^0(\sco_\piii(1))^\vee \otimes_k \sco_Y(1) 
\lra \text{T}_\piii \vb Y \lra 0\, , 
\end{gather*} 
where $\text{N}_Y := \sch om_{\sco_\piii}(\sci_Y , \sco_Y)$ is the normal bundle 
of $Y$ in $\piii$ (and $\text{T}_Y$ is the tangent bundle of $Y$).  
Notice that the map $\tH^1(\sco_Y) \ra 
\tH^0(\sco_\piii(1))^\vee \otimes \tH^1(\sco_Y(1))$ is the dual of the 
multiplication map $\mu \colon \tH^0(\sco_\piii(1)) \otimes 
\tH^0(\omega_Y(-1)) \ra \tH^0(\omega_Y)$. It follows that if $\mu$ is 
injective then $\tH^1(\text{T}_\piii \vb Y) = 0$ which implies that 
$\tH^1(\text{N}_Y) = 0$ (hence $H_{d,g}$ is nonsingular, of (local) dimension 
$\h^0(\text{N}_Y) = \chi(\text{N}_Y) = 4d$, at $[Y]$) and that the connecting 
morphism $\partial \colon \tH^0(\text{N}_Y) \ra \tH^1(\text{T}_Y)$ is 
surjective. But $\tH^0(\text{N}_Y)$ is the tangent space of $H_{d,g}$ at $[Y]$ 
and $\partial$ is the Kodaira-Spencer map of the universal family of curves 
over $H_{d,g}$ at $[Y]$. One deduces that the restriction of the natural map 
$H_{d,g} \ra \mathcal{M}_g$ to a neighbourhood of $[Y]$ is dominant. 

In order to use the intrinsic geometry of $Y$, one assumes that $Y$ is 
\emph{linearly normal} (that is, $\tH^0(\sco_\piii(1)) \izo \tH^0(\sco_Y(1))$ 
or, equivalently, $\tH^1(\sci_Y(1)) = 0$). In this case, $\mu$ is 
injective if and only if the multiplication map $\mu_0(Y) \colon 
\tH^0(\sco_Y(1)) \otimes \tH^0(\omega_Y(-1)) \ra \tH^0(\omega_Y)$ \emph{is 
injective}. Notice that $\h^1(\omega_Y(-1)) = \h^0(\sco_Y(1)) = 4$ hence, 
by Riemann-Roch on $Y$, $\h^0(\omega_Y(-1)) = g - d + 3$ and the injectivity 
of $\mu_0(Y)$ implies that the Brill-Noether number $\rho(g,3,d) = g - 
4(g - d + 3) = 4d - 3g - 12$ is non-negative. One assumes that $d$ is the 
least integer satisfying this condition.       

\vskip2mm

\noindent 
{\bf 1.2. Space curves and vector bundles with {\boldmath{$c_1 = 5$}}.}\quad 
If $\omega_Y(-1)$ is globally generated then, putting $r := 1 + 
\h^0(\omega_Y(-1))$, any epimorphism $\delta \colon (r - 1)\sco_\piii \ra 
\omega_Y(-1)$ defined by a $k$-basis of $\tH^0(\omega_Y(-1))$ determines, 
according to Serre's method of extensions, an extension$\, :$ 
\[
0 \lra (r - 1)\sco_\piii \lra E \lra \sci_Y(5) \lra 0\, , 
\]    
with $E$ locally free of rank $r$, such that, dualizing the extension, one 
gets the exact sequence$\, :$ 
\[
0 \lra \sco_\piii(-5) \lra E^\vee \lra (r - 1)\sco_\piii 
\overset{\delta}{\lra} \omega_Y(-1) \lra 0\, . 
\]
One deduces, from the last exact sequence, that $\tH^i(E^\vee) = 0$, $i = 0,\, 
1$. The Chern classes of $E$ are $c_1 = 5$, $c_2 = d$, $c_3 = \text{deg}\, 
\omega_Y(-1) = 2g - 2 - d$. $Y$ is linearly normal if and only if 
$\tH^1(E(-4)) = 0$ (which, by Serre duality, is equivalent to $\tH^2(E^\vee) 
= 0$) and, in this case, $E$ has rank $r = g - d + 4$. Under these 
assumptions, $\mu_0(Y)$ is injective if and only if $\tH^0(E^\vee(1)) = 0$.  

Conversely, let $E$ be a vector bundle of rank $r$, with  
$c_1 = 5$, and such that $\tH^i(E^\vee) = 0$, $i = 0,\, 1$. Let $W$ be an  
$(r-1)$-dimensional vector subspace of $\tH^0(E)$ such that the degeneracy 
locus of the evaluation morphism $W \otimes_k \sco_\piii \ra E$ is a 
nonsingular curve $Y$. Then the Eagon-Northcott complex of this morphism 
provides an exact sequence$\, :$ 
\[
0 \lra W \otimes_k \sco_\piii \ra E \lra \sci_Y(5) \lra 0\, , 
\] 
which, by dualization, defines an epimorphism $\delta \colon W^\vee \otimes_k 
\sco_\piii \ra \omega_Y(-1)$ with $\tH^0(\delta)$ bijective. One gets, in 
this case, a family of nonsingular space curves parametrized by an open 
subset of the Grassmannian ${\mathbb G}_{r-1}(\tH^0(E))$ of 
$(r-1)$-dimensional vector subspaces of $\tH^0(E)$. 

Now, if $E$ varies in a family with irreducible, rational base, then, in 
order to get a family of nonsingular space curves with irreducible, rational 
base, one has to assume that $\h^0(E)$ is constant in the family of vector 
bundles and that the evaluation morphisms $\tH^0(E) \otimes_k \sco_\piii \ra E$ 
can be patched together to a morphism of vector bundles on the total space of 
the family. This can be accomplished by assuming that $\tH^1(E) = 0$, for 
any bundle $E$ in the family, by Grothendieck's cohomology and base change 
theorem (see, for example, Tengan \cite[Cor.~3.1]{te}). 
Notice, also, that 
$\tH^2(E) \simeq \tH^2(\sci_Y(5)) \simeq \tH^1(\sco_Y(5))$, that  
$\tH^1(\sco_Y(5)) = 0$ if $5d \geq 2g  - 1$, 
and that $\tH^3(E) \simeq \tH^0(E^\vee(-4))^\vee = 0$ hence, by Riemann-Roch, 
$\h^0(E) = 2g - 6d + 58$ if $\tH^1(E) = 0$. Finally, the condition 
$\tH^1(E) = 0$ is equivalent to $\tH^1(\sci_Y(5)) = 0$. 

\vskip2mm 

\noindent 
{\bf 1.3. Linear monads.}\quad 
Let $E$ be a vector bundle of rank $r = g - d + 4$ and Chern classes 
$c_1 = 5$, $c_2 = d$, $c_3 = 2g - 2 - d$. Applying Cor.~\ref{C:linearmonad} 
(and Remark~\ref{R:linearmonad}) from Appendix~\ref{A:monads} to the 
vector bundle $F := E(-2)$ (as a matter of notation, this $F$ is what Chang 
and Ran denote by $E$) one gets that $E(-2)$ is the cohomology sheaf of 
a monad of the form$\, :$ 
\[
0 \lra \rho \sco_\piii(-1) \overset{\beta}{\lra}  \sigma \sco_\piii 
\overset{\alpha}{\lra} \tau \sco_\piii(1) \lra 0  
\]   
if and only if $\tH^0(E(-3)) = 0$, $\tH^0(E^\vee(1)) = 0$, $\tH^1(E(-4)) = 0$ 
and $\tH^1(E^\vee) = 0$. Actually, the last condition is, in our case, a 
consequence of the first three because $\chi(E^\vee) = 0$ (one can use the 
convenient form of Riemann-Roch stated in Thm.~2.3 from \cite{ha}) and the 
first three conditions imply that $\tH^i(E^\vee) = 0$ for $i \neq 1$.   
According to the last part of Cor.~\ref{C:linearmonad}, one must have  
$\tau = - \chi(E(-3)) = \chi(E^\vee(-1)) = 2d - g - 9$, 
$\rho = -\chi(E^\vee(1)) = 4d - 3g - 12$ and $\sigma = r - \tau - \rho = 
5d - 3g - 17$. If $Y$ is a nonsigular curve that can be 
described as the dependency locus of $r - 1$ global sections of $E$ then the 
above conditions on $E$ are equivalent to$\, :$ $\tH^0(\sci_Y(2)) = 0$, 
$\mu_0(Y)$ injective and $\tH^1(\sci_Y(1)) = 0$ (that is, $Y$ linearly 
normal).   

It is easy to see that the monads of the above form with the property that  
$\tH^0(\beta^\vee(1)) \colon \tH^0(\sigma \sco_\piii(1)) \ra 
\tH^0(\rho \sco_\piii(2))$ is surjective can be put together into a family with 
irreducible rational base. If $E(-2)$ is the cohomology sheaf of a monad of the 
above form, then $\tH^0(\beta^\vee(1))$ is surjective if and only if 
$\tH^1(E^\vee(3)) = 0$. In terms of the curve $Y$, the last condition is 
equivalent to the fact that the multiplication map $\tH^0(\omega_Y(-1)) \otimes 
\tH^0(\sco_\piii(3)) \ra \tH^0(\omega_Y(2))$ is surjective. Notice that this 
condition implies that $\omega_Y(-1)$ is globally generated. Notice also that, 
by Lemma~\ref{L:5ora2o(1)} from Appendix~\ref{A:monads}, the condition 
$\tH^0(\beta^\vee(1))$ surjective is automatically satisfied if $\rho \leq 2$. 

\vskip2mm

\noindent
{\bf 1.4. The approach of Chang and Ran.}\quad 
Taking into account what has been said, in order to show that 
the moduli space $\mathcal{M}_g$ is unirational it suffices to prove the 
existence of nonsingular, connected space curves $Y \subset \piii$, of genus 
$g$ and degree $d$, with the folowing properties$\, :$   
\begin{enumerate} 
\item[(L)] $Y$ \emph{is linearly normal}$\, ;$ 
\item[(M)] $\mu_0(Y)$ \emph{is injective}$\, ;$ 
\item[(R)] $\tH^0(\sci_Y(2)) = 0$ \emph{and} $\tH^1(\sci_Y(5)) = 0$$\, ;$ 
\item[(S)] \emph{The multiplication map} $\tH^0(\omega_Y(-1)) \otimes 
\tH^0(\sco_\piii(3)) \ra \tH^0(\omega_Y(2))$ \emph{is surjective}. 
\end{enumerate}
Actually, Chang and Ran use, instead of (R), the stronger condition asserting 
that "$Y$ \emph{has maximal rank}" (which, in the cases $10 \leq g \leq 13$, 
means that $\tH^0(\sci_Y(4)) = 0$ and $\tH^1(\sci_Y(5)) = 0$) although they are 
aware of the fact that the above weaker condition is sufficient (see 
\cite[Remark~3.1]{cr}). They prove the above existence result by starting 
with a curve of smaller degree and genus, having some "good properties", 
and successively attaching 4-secant conics to it. They show that, if one is 
careful enough, the resulting reducible curve has the above properties and 
can be smoothed. This approach is based on results of Sernesi~\cite{se2} and, 
in particular, on his results showing the injectivity of the $\mu_0$-map 
of a reducible curve obtained by attaching 4-secant conics. 

\vskip2mm 

\noindent 
{\bf 1.5. An alternative approach.}\quad  
The first three subsections above show that the unirationality of 
$\mathcal{M}_g$, $8 \leq g \leq 13$, is a consequence of the following$\, :$  

\begin{theorem}\label{T:main} 
Let $g$ be an integer with $8 \leq g \leq 13$, and let $d$ be the least 
integer for which $\rho := 4d - 3g -12 \geq 0$. Consider, also, the integers 
$\sigma := 5d - 3g - 17$ and $\tau := 2d - g - 9$. Then there exist vector 
bundles $E$ on $\piii$, of rank $r := g - d + 4$, subject to the following 
conditions$\, :$ 

\emph{(a)} $F := E(-2)$ is the cohomology sheaf of a monad of the form$\, :$ 
\begin{equation}\label{E:monade(-2)} 
0 \lra \rho \sco_\piii(-1) \overset{\beta}{\lra}  \sigma \sco_\piii 
\overset{\alpha}{\lra} \tau \sco_\piii(1) \lra 0\, , 
\end{equation} 
with ${\fam0 H}^0(\beta^\vee(1))$ surjective$\, ;$ 

\emph{(b)} ${\fam0 H}^1(E) = 0$$\, ;$ 

\emph{(c)} The dependency scheme of $r - 1$ general global sections of $E$ is 
a nonsingular curve.     
\end{theorem} 

In order to prove the theorem, we consider bundles $E$ with the property that 
their duals $E^\vee$ can be realized as extensions$\, :$ 
\begin{equation}\label{E:aeveescic}
0 \lra A \lra E^\vee \lra \sci_C \lra 0\, , 
\end{equation} 
where, for $8 \leq g \leq 11$ (when $d = g + 1$ and $r = 3$), $A = 
\sco_\piii(-3) \oplus \sco_\piii(-2)$ and $C$ is a rational curve of degree 
$g - 5$, while, for $g = 12,\, 13$ (when $d = g$ and $r = 4$), $A = 
\sco_\piii(-3) \oplus 2\sco_\piii(-1)$ and $C$ is a rational curve of degree 
$g - 7$. These bundles occured naturally in our systematic study of 
globally generated vector bundles with $c_1 = 5$ on $\piii$ from \cite{acm4}.   

The extensions \eqref{E:aeveescic} can be constructed by elementary 
transformations as follows$\, :$ assume that $A = 
\bigoplus_{i=0}^m\sco_\piii(-a_i)$, $a_0 \geq \ldots \geq a_m$, and that $C$ is a 
nonsingular connected curve on a nonsingular surface $X \subset \piii$ of 
degree $d$. Let $\gamma$ be a global section of $\sco_X[C]$ whose zero divisor 
is $C$ and let $\delta_i$ be a global section of $\sco_X(d-a_i) \otimes 
\sco_X[C]$, with zero divisor $D_i$, $i = 0, \ldots , m$, such that $C \cap 
D_0 \cap \ldots \cap D_m = \emptyset$. Let $\e \colon \sco_\piii(d) \oplus 
A^\vee \ra \sco_X(d) \otimes \sco_X[C]$ be the epimorphism defined by $\gamma, 
\delta_0 , \ldots , \delta_m$ and let $E$ be the kernel of $\e$. 

Recall the following simple observation$\, :$ if $\phi \colon \scf_0 \oplus 
\scf_1 \ra \scg$ is a morphism of sheaves, with components $\phi_i \colon 
\scf_i \ra \scg$, $i = 0,\, 1$, then one has an exact sequence$\, :$ 
\[
0 \lra \Ker \phi_0 \lra \Ker \phi \lra \scf_1 
\overset{{\overline \phi}_1}{\lra} \Cok \phi_0 \lra \Cok \phi \lra 0\, ,  
\]   
where ${\overline \phi}_1$ is the composite morphism $\scf_1  
\overset{\phi_1}{\lra} \scg \ra \Cok \phi_0$. Using this observation and the 
adjunction formula one gets that $(\sco_X(d) \otimes \sco_X[C]) \vb C \simeq 
\omega_C(4)$ and an exact sequence$\, :$ 
\[
0 \lra \sco_\piii \lra E \lra A^\vee \overset{{\overline \delta}}{\lra} 
\omega_C(4) \lra 0\, , 
\] 
with $\overline \delta$ defined by $\delta_i \vb C$, $i = 0, \ldots , m$. 
Let $\widetilde \sck$ be the kernel of $\overline \delta$. Dualizing the 
exact sequence$\, :$ 
\begin{equation}\label{E:oescktilde} 
0 \lra \sco_\piii \lra E \lra {\widetilde \sck} \lra 0\, , 
\end{equation}
one gets the extension \eqref{E:aeveescic} above. 

We shall use two filtrations of $\widetilde \sck$. Firstly, applying the above 
observation, one gets an exact sequence$\, :$ 
\begin{equation}\label{E:ica1scktilde}  
0 \ra \sci_C(a_0) \ra {\widetilde \sck} \ra 
{\textstyle \bigoplus}_{i=1}^m\sco_\piii(a_i) \xra{{\overline \delta}^\prime} 
(\sco_X(d) \otimes \sco_X[C]) \vb D_0 \cap C \ra 0\, ,  
\end{equation}
where ${\overline \delta}^\prime$ is defined by $\delta_i \vb D_0 \cap C$, 
$i = 1 , \ldots , m$.  

Secondly, let $\sck$ be the kernel of the morphism $\delta \colon A^\vee \ra 
\sco_X(d) \otimes \sco_X[C]$ defined by $\delta_0, \ldots , \delta_m$. Applying 
the Snake Lemma to the diagram$\, :$ 
\[
\SelectTips{cm}{12}\xymatrix{0\ar[r] & 0\ar[r]\ar[d] & 
{A^\vee}\ar @{=}[r]\ar[d]^{\delta} & {A^\vee}\ar[r]\ar[d]^{{\overline \delta}} 
& 0\\ 
0\ar[r] & {\sco_X(d)}\ar[r]^-{\gamma} & {\sco_X(d) \otimes \sco_X[C]}\ar[r] & 
{\omega_C(4)}\ar[r] & 0}
\]
one gets an exact sequence$\, :$ 
\begin{equation}\label{E:sckscktilde} 
0 \lra \sck \lra {\widetilde \sck} \lra \sci_{D_0 \cap \ldots \cap D_m , X}(d) 
\lra 0\, . 
\end{equation} 
Notice that if $D_i = \Delta_i + B$, $i = 0 , \ldots , m$, such that the 
scheme $W := \Delta _0 \cap \ldots \cap \Delta_m$ is 0-dimensional (or empty), 
then $\sci_{D_0 \cap \ldots \cap D_m , X} \simeq \sci_{W , X} \otimes \sco_X[-B]$.  
On the other hand, applying the above observation to $\delta$ one obtains an 
exact sequence$\, :$ 
\begin{equation}\label{E:scoscksckprime}  
0 \ra \sco_\piii(a_0 - d) \ra \sck \ra 
{\textstyle \bigoplus}_{i=1}^m\sco_\piii(a_i) \overset{\delta^\prime}{\lra} 
(\sco_X(d) \otimes \sco_X[C]) \vb D_0\, , 
\end{equation}
where $\delta^\prime$ is defined by $\delta_i \vb D_0$, $i = 1 , \ldots , m$. 
Notice that if $D_i = \Delta_i + B$, $i = 0 , \ldots , m$, such that $\Delta_0 
\cap B$ consists of finitely many points then the last exact sequence induces 
an exact sequence$\, :$ 
\begin{equation}\label{E:scoscksckprimebis}  
0 \ra \sco_\piii(a_0 - d) \ra \sck \ra 
{\textstyle \bigoplus}_{i=1}^m\sco_\piii(a_i) \overset{\delta^\secund}{\lra} 
(\sco_X(d) \otimes \sco_X[C]) \vb \Delta_0\, , 
\end{equation}
where $\delta^\secund$ is defined by $\delta_i \vb \Delta_0$, $i = 1 , \ldots , 
m$. The reason is that $\sco_{D_0}$ embeds into $\sco_{\Delta_0} \oplus \sco_B$ 
and $\delta_i$ vanishes on $B$, $i = 1 , \ldots , m$.  

For our purposes, if $C$ is a rational curve of degree 6 (resp., 5) we take 
$X$ to be a cubic surface which is the blow-up $\pi \colon X \ra \pii$ of 
$\pii$ in six general points $P_1 , \ldots , P_6$, embedded in $\piii$ such 
that $\sco_X(1) \simeq \pi^\ast\sco_\pj(3) \otimes \sco_X[-E_1 - \ldots -E_6]$, 
where $E_i := \pi^{-1}(P_i)$, and we take $C$ to be the strict transform of a 
nonsingular conic $\overline{C} \subset \pii$ containing none of the points 
$P_1 , \ldots , P_6$ (resp., containing $P_1$ but none of the points $P_2 , 
\ldots , P_6$), while if $C$ is a rational curve of degree 4 or 3, we take 
$X$ to be a quadric surface containing $C$. The divisors $D_0 , \ldots , D_m$ 
will be specified during the proof.   

\vskip2mm 

Now, using the above construction, the properties (a) and (b) from the 
conclusion of the theorem can be easily checked. Actually, for $g = 12,\, 13$, 
there is a slight technical complication due to the fact that the bundles 
$E$ constructed as above have the property that $E(-2)$ is the cohomology 
sheaf of a minimal monad of the form$\, :$ 
\begin{equation}\label{E:monade(-2)1} 
0 \lra (\rho + 2) \sco_\piii(-1) \overset{\beta^\prim}{\lra}  
\sigma \sco_\piii \oplus 2\sco_\piii(-1) 
\overset{\alpha^\prim}{\lra} \tau \sco_\piii(1) \lra 0\, , 
\end{equation} 
with $\tH^0(\beta^{\prim \vee}(1))$ surjective. This is, however, harmless  
because the monads of the form \eqref{E:monade(-2)1} with 
$\tH^0(\beta^{\prim \vee}(1))$  surjective can be put together into a family 
with irreducible base. For a general monad of this type, the component 
$(\rho + 2)\sco_\piii(-1) \ra 2\sco_\piii(-1)$ of $\beta^\prim$ is surjective 
hence the cohomology sheaf of the monad is isomorphic to the cohomology 
sheaf of a monad of the form \eqref{E:monade(-2)} with $\tH^0(\beta^\vee(1))$ 
surjective. Moreover, the conditions (b) and ``(b) + (c)'' are open 
conditions in the family of vector bundles $E$ with the property that $E(-2)$ 
is the cohomology sheaf of a monad of the form \eqref{E:monade(-2)1} (because, 
in this case, $\tH^i(E) = 0$ for $i \geq 2$). 

\vskip2mm 

The non-trivial part of the proof is the verification of condition (c). 
We verify, for $g \leq 12$, the stronger condition ``$E$ \emph{is globally 
generated}''. Actually, $E$ is 0-regular for $8 \leq g \leq 10$, while, for 
$g = 11,\, 12$, $E$ is 1-regular (this follows from (a) and (b)) and we show 
that ``\emph{the multiplication map} $\tH^0(E) \otimes S_1 \ra \tH^0(E(1))$ 
\emph{is surjective}''. Here $S_1 := \tH^0(\sco_\piii(1))$ is the space of 
linear forms on $\piii$. 

\vskip2mm 

On the other hand, if $g = 13$ then $d = 13$, $E$ has rank 4 and, if it 
satisfies (a) and (b), one has $\h^0(E) = 6$ hence the degeneracy locus of the 
evaluation morphism $6\sco_\piii \ra E$ is non-empty. The best one can hope for 
in this case is that ``\emph{the evaluation morphism} $\tH^0(E) \otimes_k 
\sco_\piii \ra E$ \emph{is an epimorphism except at finitely many points where 
it has corank} 1'' (this would, obviously, imply (c)). We were not able to 
verify this condition for the bundles $E$ constructed as above. We can, 
fortunately, show that some of the bundles $E$ constructed as above satisfy the 
weaker condition asserting that ``\emph{the evaluation morphism of} $E$ 
\emph{has corank at most} 1 \emph{at every point and its degeneracy scheme is 
a curve contained in a nonsingular surface in} $\piii$''. 
According to some results of Martin-Deschamps and Perrin \cite{mdp}, recalled 
in Lemma~\ref{L:dependencyloci} from Appendix~\ref{A:monads}, this condition 
still implies (c).  



\section{The Cases $8 \leq g  \leq 11$}\label{S:g=11}

In these cases, $d = g + 1$, $r = 3$ and the monads from 
Theorem~\ref{T:main}(a) have the form$\, :$ 
\begin{equation}\label{E:monadgleq11}  
0 \lra (g - 8)\sco_\piii(-1) \overset{\beta}{\lra} 2(g - 6)\sco_\piii 
\overset{\alpha}{\lra} (g - 7)\sco_\piii(1) \lra 0\, . 
\end{equation}
Using the notation from Subsection~1.5, one has $A = 
\sco_\piii(-3) \oplus \sco_\piii(-2)$ and $C \subset \piii$ is a nonsingular 
rational curve of degree $g - 5$, contained in a nonsingular surface $X 
\subset \piii$ of degree 3 or 2. We shall specify, now, the divisors $D_0$ 
and $D_1$. 

\vskip2mm 

\noindent 
$\bullet$\quad If $g = 11$ then $C$ has degree 6. $X$ is the unique effective 
divisor of degree 3 in $\piii$ containing $C$ because $C$ admits six 
4-secants (namely, the strict transforms of the conics in $\pii$ containing 
five of the six points $P_1 , \ldots , P_6$) and a complete intersection 
of type $(3 , 3)$ in $\piii$ containing $C$ would contain all of these 
4-secants. Since $\h^0(\sci_C(3)) = 1$ and $\h^0(\sco_C(3)) = 19$ it follows 
that $\tH^1(\sci_C(3)) = 0$. Consider a general (nonsingular) conic 
${\overline C}_0 \subset \pii$, containing $P_1,\, P_2,\, P_3$ but none of 
the points $P_4,\, P_5,\, P_6$, and intersecting $\overline C$ (the conic  
in $\pii$ whose inverse image on $X$ is $C$) in four distinct points. The 
strict transform $C_0 \subset X$ of ${\overline C}_0$  is a twisted cubic 
curve in $\piii$. One has $C_0 \sim C - E_1 - E_2 - E_3$ as divisors on $X$. 
Put $Y := E_1 + E_2 + E_3$. By the adjunction formula, $\sco_X[C_0] \vb C_0 
\simeq \omega_{C_0}(1)$. Moreover, the restriction map $\tH^0(\sco_X(1) \otimes 
\sco_X[C_0]) \ra \tH^0((\sco_X(1) \otimes \sco_X[C_0]) \vb C_0)$ is surjective 
because its cokernel embeds into $\tH^1(\sco_X(1)) = 0$. It follows that if 
$D$ is a general member of the complete linear system $\vert \, \sco_X(1) 
\otimes \sco_X[C_0] \, \vert$ then the scheme $W := D \cap C_0$ consists of 
four \emph{general} simple points of $C_0$, none of them belonging to $C$ or 
to $Y$. We take $D_0 := C_0 + Y$ and $D_1 := D + Y$. 

\vskip2mm 

\noindent 
$\bullet$\quad If $g = 10$ then $C$ has degree 5. Since $\tH^0(\sci_C(2)) = 0$ 
(because, otherwise, $C$ would be linked by a complete intersection of type 
$(2 , 3)$ to a line and this would contradict the fact that $C$ is not 
arithmetically Cohen-Macaulay) it follows that $\h^1(\sci_C(2)) = 1$. If 
$H \subset \piii$ is a general plane then $H \cap C$ consists of five points, 
no three collinear. In this case, $\sci_{H \cap C , H}$ is 3-regular. The 
Lemma of Le Potier (see, for example, \cite[Lemma~1.22]{acm1}) implies that 
$\tH^1(\sci_C(3)) = 0$. Consider a general (nonsingular) conic 
${\overline C}_0 \subset \pii$, containing $P_1,\, P_2,\, P_3$ but none of 
the points $P_4,\, P_5,\, P_6$, and intersecting $\overline C$ (the conic  
in $\pii$ whose strict transform on $X$ is $C$) in four distinct points ($P_1$ 
being one of them). The strict transform $C_0 \subset X$ of ${\overline C}_0$  
is a twisted cubic curve in $\piii$. One has $C_0 \sim C - E_2 - E_3$ as 
divisors on $X$ and $C_0 \cap C$ consists of three simple points.  
Let $D$ be a general member of the complete linear system $\vert \, 
\sco_X(1) \otimes \sco_X[C_0] \, \vert$ such that the scheme $W := D \cap C_0$ 
consists of four \emph{general} simple points of $C_0$, none of them belonging 
to $C$ or to $E_2 \cup E_3$. We take $D_0 := C_0 + E_2 + E_3$ and $D_1 := D + 
E_2 + E_3$. 

\vskip2mm 

\noindent 
$\bullet$\quad If $g = 9$ (resp., $g = 8$) then $C$ is a divisor of type 
$(3 , 1)$ (resp., $(2 , 1)$) on a nonsingular quadric surface $X \simeq 
\pj \times \pj$ in $\piii$. In both cases, $\tH^1(\sci_C(2)) = 0$. Consider 
two lines $L_0$, $L_1$ of type $(1 , 0)$ on $X$. We take $D_0 := L_0 + L_1$ 
(resp., $D_0 = L_0$) and $D_1$ any member of the complete linear system 
$\vert \, \sco_X[C] \, \vert$ containing none of the points of $C \cap D_0$. 

\vskip2mm 

Now, in all of the above cases, $D_0 \cap C$ consists of $g - 7$ simple 
points, no three collinear and no four coplanar (recall that, for $g = 11$ 
and $g = 10$, they belong to the twisted cubic curve $C_0 \subset \piii$). 
Using the exact sequence \eqref{E:oescktilde} and the exact sequence 
\eqref{E:ica1scktilde} that becomes$\, :$ 
\[
0 \lra \sci_C(3) \lra {\widetilde \sck} \lra \sci_{C_0 \cap C}(2) \lra 0\, , 
\] 
one deduces that $\tH^0(E(-2)) = 0$ (and $\tH^0(E(-1)) = 0$ if $g = 11$) and 
that $\h^1(E(-3)) = \h^1(\sci_{D_0 \cap C}(-1)) = g - 7$. On the other hand, 
using the extension \eqref{E:aeveescic}, one gets that $\tH^0(E^\vee(1)) = 0$, 
$\tH^1(E(-4)) \simeq \tH^2(E^\vee)^\vee = 0$ and $\tH^1(E^\vee) = 0$. Moreover, 
$\h^1(E^\vee(1)) = \h^1(\sci_C(1)) = g - 8$. Cor.~\ref{C:linearmonad} from 
Appendix~\ref{A:monads} implies, now, that $F := E(-2)$ is the cohomology 
sheaf of a 
monad of the form \eqref{E:monadgleq11}. Notice, also, that $\tH^1(E^\vee(3)) 
\simeq \tH^1(\sci_C(3)) = 0$, hence $\tH^0(\beta^\vee(1))$ is surjective 
($\beta$ being the differential of the monad). Consequently, $E$  
satisfies condition (a) from Theorem~\ref{T:main}.    
Using, again, the fact that $\tH^1(\sci_C(3)) = 0$ and the exact sequences 
\eqref{E:oescktilde} and \eqref{E:ica1scktilde}, one gets that $\tH^1(E) = 0$ 
hence $E$ satisfies condition (b) from Theorem~\ref{T:main}. Moreover, for 
$g = 8,\, 9$, one has $\tH^1(E(-1)) = 0$ hence $E$ is 0-regular in those cases 
(actually, for $g = 8$ there is only one bundle $E$ satisfying the conclusion 
of Theorem~\ref{T:main}, namely $E = \Omega_\piii(3)$$\, :$ look at the monad 
\eqref{E:monadgleq11}). It, consequently, remains to show that the 
multiplication map $\tH^0(E) \otimes S_1 \ra \tH^0(E(1))$ is surjective if 
$g = 11$ and that $\tH^1(E(-1)) = 0$ if $g = 10$.  

\vskip3mm 

\noindent 
{\bf Conclusion of the proof of Theorem~\ref{T:main} for 
{\boldmath{$g = 11$}}.}\quad 
For the above choice of the divisors $D_0$ and $D_1$, the exact sequences 
\eqref{E:sckscktilde} and \eqref{E:scoscksckprimebis} become$\, :$ 
\begin{gather*} 
0 \lra \sck \lra {\widetilde \sck} \lra \sci_{Y \cup W , X}(3) \lra 0\, ,\\ 
0 \lra \sco_\piii \lra \sck \lra \sci_{C_0}(2) \lra 0\, .  
\end{gather*}
It follows that, in order to show that 
the multiplication map $\tH^0(E) \otimes S_1 \ra \tH^0(E(1))$ is surjective, it 
suffices to check that $E_1,\, E_2,\, E_3$ and $W$ satisfy the 
hypotheses of Lemma~\ref{L:3lines+4points} from Appendix~\ref{A:auxiliary} 
and that is exactly what we are going to do next. 

Since $W$ consists of four simple points on the twisted cubic curve $C_0$ 
it is not contained in a plane. We assert that, for $1 \leq l \leq 3$, 
$(Y \setminus E_l) \cup C_0$ is not contained in a quadric surface. 
\emph{Indeed}, if it would be contained then the surface would be nonsingular, 
isomorphic to $\pj \times \pj$, $C_0$ would be a divisor of type $(2 , 1)$ on 
this surface and the two components of $Y \setminus E_l$ would be divisors 
of type $(1 , 0)$ (because the intersection of each of them with $C_0$ is a 
simple point). But this would \emph{contradict} the fact that 
$(Y \setminus E_l) \cup C_0$ is contained in an irreducible cubic surface. 

It follows that the restriction map $\tH^0(\sci_{Y \setminus E_l}(2)) \ra 
\tH^0(\sco_{C_0}(2))$ is injective, $l = 1,\, 2,\, 3$. Since $W$ consists 
of four \emph{general} points of $C_0$, one can, consequently, assume that 
$\tH^0(\sci_{(Y \setminus E_l) \cup W}(2)) = 0$, $l = 1,\, 2,\, 3$. Moreover, 
one can assume that none of the points of $W$ belongs to the quadric 
surface containing $Y$. This completes the verification of the hypotheses of 
Lemma~\ref{L:3lines+4points} and, with it, the proof of the assertion that the 
multiplication map $\tH^0(E) \otimes S_1 \ra \tH^0(E(1))$ is 
surjective.\hfill $\Box$  


\vskip3mm 
 
\noindent 
{\bf Conclusion of the proof of Theorem~\ref{T:main} for 
{\boldmath{$g = 10$}}.}\quad 
For the above choice of the divisors $D_0$ and $D_1$, the exact sequences 
\eqref{E:sckscktilde} and \eqref{E:scoscksckprimebis} become$\, :$ 
\begin{gather*} 
0 \lra \sck \lra {\widetilde \sck} \lra \sci_{E_2 \cup E_3 \cup W , X}(3) \lra 
0\, ,\\ 
0 \lra \sco_\piii \lra \sck \lra \sci_{C_0}(2) \lra 0\, . 
\end{gather*}
As we saw in the above proof of the case $g = 11$ of Theorem~\ref{T:main},    
one can assume that $E_2 \cup E_3 \cup W$ is contained in no quadric surface 
in $\piii$. This implies that $\tH^1(\sci_{E_2 \cup E_3 \cup W}(2)) = 0$ (because 
$\h^0(\sco_{E_2 \cup E_3 \cup W}(2)) = 10$) hence $\tH^1(E(-1)) = 0$.\hfill $\Box$  


\section{The Case $g = 13$}\label{S:g=13} 

\noindent 
{\bf Proof of Theorem~\ref{T:main} for {\boldmath{$g = 13$}}.}\quad 
In this case, $d = 13$, $r = 4$, and the monad from the statement of 
Theorem~\ref{T:main}(a) has the form$\, :$ 
\begin{equation}\label{E:monadg=13}
0 \lra \sco_\piii(-1) \overset{\beta}{\lra} 9\sco_\piii 
\overset{\alpha}{\lra} 4\sco_\piii(1) \lra 0\, . 
\end{equation} 
Using the notation from Subsection~1.5, one has $A = 
\sco_\piii(-3) \oplus 2\sco_\piii(-1)$ and $C$ is a nonsingular 
rational curve of degree 6 contained in a nonsingular cubic surface $X$ in 
$\piii$. We shall specify, now, the divisors $D_0$, $D_1$, $D_2$. Let $C_0 
\subset X$ be a twisted cubic curve which is the strict transform of a general 
(nonsingular) conic ${\overline C}_0 \subset \pii$, containing $P_1,\, P_2,\, 
P_3$ but none of the points $P_4,\, P_5,\, P_6$, and intersecting 
$\overline C$ (the conic in $\pii$ whose inverse image on $X$ is $C$) in four 
distinct points. One has $C_0 \sim C - Y$, where $Y := E_1 + E_2 + E_3$. 
Let $\gamma_0$ be a global section of $\sco_X[C_0]$ whose zero divisor is 
$C_0$. Complete $\gamma_0$ to a $k$-basis $\gamma_0,\, \gamma_1,\, \gamma_2$ 
of $\tH^0(\sco_X[C_0])$ and let $C_i \subset X$ be the zero divisor of 
$\gamma_i$, $i = 1,\, 2$. 
Let, finally, $D$ be a general member of the linear 
system $\vert \, \sco_X(2) \, \vert$ such that $W := D \cap C_0$ consists of 
six \emph{general} simple points of $C_0$, not belonging to $C \cup Y$. We 
take $D_0 := C_0 + Y$, $D_1 := D + C_1 + Y$, $D_2 := D + C_2 + Y$. 

\vskip2mm 

\noindent 
{\bf Claim 1.}\quad $F := E(-2)$ \emph{is the cohomology sheaf of a monad of 
the form}$\, :$ 
\[
0 \lra 3\sco_\piii(-1) \overset{\beta^\prim}{\lra} 9\sco_\piii \oplus 
2\sco_\piii(-1) \overset{\alpha^\prim}{\lra} 4\sco_\piii(1) \lra 0\, ,  
\] 
\emph{with} $\tH^0(\beta^{\prim \vee}(1))$ \emph{surjective}. 

\vskip2mm 

\noindent 
\emph{Indeed}, the exact sequence \eqref{E:ica1scktilde} becomes, now$\, :$ 
\begin{equation}\label{E:ica1scktildeg=13}
0 \lra \sci_C(3) \lra {\widetilde \sck} \lra 2\sco_\piii(1) 
\overset{{\overline \delta}^\prim}{\lra} (\sco_X(3) \otimes \sco_X[C]) \vb 
C_0 \cap C \lra 0\, , 
\end{equation} 
where ${\overline \delta}^\prim$ is defined by $\delta_i \vb C_0 \cap C$, $i = 
1,\, 2$. Taking into account the exact sequence \eqref{E:oescktilde},   
one deduces immediately that $\tH^0(E(-2)) = 0$ (one can, actually, 
show that no non-trivial linear combination of $\delta_1 \vb C_0$ and 
$\delta_2 \vb C_0$ vanishes on $C_0 \cap C$ hence $\tH^0(E(-1)) = 0$) and that 
$\h^1(E(-3)) = 4$. On the other hand, using the extension \eqref{E:aeveescic}, 
one gets that $\h^0(E^\vee(1)) = 2$, $\tH^1(E(-4)) \simeq \tH^2(E^\vee)^\vee = 
0$, $\tH^1(E^\vee) = 0$, $\tH^1(E^\vee(-1)) = 0$ and $\h^1(E^\vee(1)) \simeq 
\h^1(\sci_C(1)) = 3$. Lemma~\ref{L:quasilinearmonad} from 
Appendix~\ref{A:monads} shows, now, that $F := E(-2)$ is the cohomology of a 
monad of the above form. Since $\tH^1(E^\vee(3)) \simeq \tH^1(\sci_C(3)) = 0$ 
it follows that  $\tH^0(\beta^{\prim \vee}(1))$ is surjective. 

\vskip2mm 

\noindent 
{\bf Claim 2.}\quad $\tH^1(E) = 0$.   

\vskip2mm 

\noindent 
\emph{Indeed}, since the morphism ${\overline \delta}^\prim$ from the exact 
sequence \eqref{E:ica1scktildeg=13} is an epimorphism, it follows that a 
general linear combination $a_1\delta_1 + a_2\delta_2$ vanishes at no point 
of $C_0 \cap C$. One deduces that the composite morphism$\, :$ 
\[
\sco_\piii \xra{\left(\begin{smallmatrix} a_1\\ a_2 \end{smallmatrix}\right)} 
2\sco_\piii \xra{{\overline \delta}^\prim(-1)} (\sco_X(2) \otimes \sco_X[C]) \vb 
C_0 \cap C  
\]  
is an epimorphism hence its kernel is isomorphic to $\sci_{C_0 \cap C}$. Since 
$C_0 \cap C$ consists of four points that are not coplanar, one has 
$\tH^1(\sci_{C_0 \cap C}(1)) = 0$. One deduces that 
$\tH^0({\overline \delta}^\prim)$ is surjective and this implies that 
$\tH^1(E) = 0$. 

\vskip2mm 

It remains to verify condition (c) from Thm.~\ref{T:main}. Recall, for this 
purpose, the exact sequences \eqref{E:sckscktilde} and 
\eqref{E:scoscksckprimebis} from Subsection~1.5 that become, 
in our case$\, :$ 
\begin{gather*} 
0 \lra \sck \lra {\widetilde \sck} \lra \sci_{Y \cup W , X}(3) \lra 0\, ,\\ 
0 \lra \sco_\piii \lra \sck \lra 2\sco_\piii(1) \overset{\gamma^\prim}{\lra}  
(\sco_X(1) \otimes \sco_X[C_0]) \vb C_0 \lra 0\, , 
\end{gather*} 
where $\gamma^\prim$ is defined by $\gamma_i \vb C_0$, $i = 1, \, 2$. 
One deduces, from the second exact sequence, that $\sck$ is 0-regular. Taking 
into account the exact sequence \eqref{E:oescktilde}, one gets that the 
cokernel of the evaluation morphism of $E$ is isomorphic to the cokernel of 
the evaluation morphism of $\sci_{Y \cup W , X}(3)$.     

\vskip2mm 

\noindent 
{\bf Claim 3.}\quad \emph{There exists a unique divisor} $\Delta$ \emph{in the 
complete linear system} $\vert \, \sco_X(3) \otimes \sco_X[-Y] \, \vert$ 
\emph{such that} $\Delta \cap C_0 = W$ \emph{as schemes}. 

\vskip2mm 

\noindent
\emph{Indeed}, the restriction map 
$\tH^0(\sco_X(3) \otimes \sco_X[-Y]) \ra 
\tH^0((\sco_X(3) \otimes \sco_X[-Y]) \vb C_0)$ is bijective because  
its kernel is $\tH^0(\sco_X(3) \otimes \sco_X[-C])$ and its 
cokernel embeds into $\tH^1(\sco_X(3) \otimes \sco_X[-C])$ and these cohomology 
groups are both zero because $\h^0(\sci_C(3)) = 1$ and $\tH^1(\sci_C(3)) = 0$. 
One uses, now, the fact that $(\sco_X(3) \otimes \sco_X[-Y]) \vb C_0$ is 
a line bundle of degree 6 on $C_0 \simeq \pj$ and $W$ is an effective divisor 
of degree 6 on $C_0$. 

\vskip2mm 
  
One deduces, from Claim 3, that the cokernel of the evaluation morphism of 
$\sci_{Y \cup W , X}(3) \simeq \sci_{W , X}(3) \otimes \sco_X[-Y]$ is isomorphic to 
$\sci_{W , \Delta} \otimes \scl$, where $\scl$ is the restriction of 
$\sco_X(3) \otimes \sco_X[-Y]$ to $\Delta$. Moreover, since $W$ consists 
of six \emph{simple} points of $C_0$ and since $\Delta \cap C_0 = W$ as 
schemes, it follows that the points of $W$ are \emph{nonsingular} points of 
$\Delta$, hence $\sci_{W , \Delta} \otimes \scl$ is an invertible 
$\sco_\Delta$-module. The results of Martin-Deschamps and Perrin \cite{mdp} 
recalled in Lemma~\ref{L:dependencyloci} and Remark~\ref{R:dependencyloci} from 
Appendix~\ref{A:monads} imply, now, that the dependency locus of three general 
global sections of $E$ is a nonsingular curve in $\piii$.\hfill $\Box$  


\section{The Case $g = 12$}\label{S:g=12}

\noindent 
{\bf Proof of Theorem~\ref{T:main} for {\boldmath{$g = 12$}}.}\quad  
In this case, $d = 12$, $r = 4$, and the monads from the statement of 
Theorem~\ref{T:main} are of the form$\, :$ 
\begin{equation}\label{E:monadg=12} 
0 \lra 0 \lra 7\sco_\piii \overset{\alpha}{\lra} 3\sco_\piii(1) \lra 0\, .   
\end{equation}
Using the notation from Subsection~1.5, one has $A = 
\sco_\piii(-3) \oplus 2\sco_\piii(-1)$ and $C$ is a nonsingular 
rational curve of degree 5 contained in a nonsingular cubic surface $X$ in 
$\piii$. We shall specify, now, the divisors $D_0$, $D_1$, $D_2$. Let $C_0 
\subset X$ be a twisted cubic curve which is the strict transform of a general 
(nonsingular) conic ${\overline C}_0 \subset \pii$, containing $P_1,\, P_2,\, 
P_3$ but none of the points $P_4,\, P_5,\, P_6$, and intersecting 
$\overline C$ (the conic in $\pii$ whose strict transform on $X$ is $C$) in 
four distinct points (including $P_1$$\, ;$ it follows that $C_0 \cap C$ 
consists of three simple points). One has $C_0 \sim C - E_2 - E_3$. Let 
$\gamma_i$, $i= 0,\, 1,\, 2$, $C_1$, $C_2$, $D$ and $W := D \cap C_0$ be as at 
the beginning of the above proof of the case $g = 13$. We take $D_0 = C_0 + 
E_2 + E_3$ and $D_i = D + C_i + E_2 + E_3$, $i = 1,\, 2$. 

\vskip2mm 

One shows easily, as in Claim~1 and Claim~2 of the above proof of the case 
$g = 13$, that $F := E(-2)$ is the cohomology sheaf of a monad of the 
form$\, :$ 
\[
0 \lra 2\sco_\piii(-1) \overset{\beta^\prim}{\lra} 7\sco_\piii \oplus 
2\sco_\piii(-1) \overset{\alpha^\prim}{\lra} 3\sco_\piii(1) \lra 0\, ,  
\]
with $\tH^0(\beta^{\prim \vee}(1))$ surjective, and that $\tH^1(E) = 0$. (One 
can, moreover, show that $\tH^0(E(-1)) = 0$). It thus remains to show that the 
multiplication map $\tH^0(E) \otimes S_1 \ra \tH^0(E(1))$ is surjective. 
This is equivalent to the fact that the multiplication map 
$\tH^0(\sci_{E_2 \cup E_3 \cup W}(3)) \otimes S_1 \ra 
\tH^0(\sci_{E_2 \cup E_3 \cup W}(4))$ is surjective (one uses arguments similar to 
those used in the above proof of the case $g = 13$). 

Recall that $W$ consists of six \emph{general} simple points of $C_0$. 
As we saw in the proof of the case $g = 11$ of Theorem~\ref{T:main} from 
Section~\ref{S:g=11}, there is no quadric surface in $\piii$ containing 
$E_2 \cup E_3 \cup C_0$. On the other hand, there is a unique quadric surface 
$Q_i$ in $\piii$ containing $E_i \cup C_0$ (because $E_i$ intersects $C_0$ 
in one simple point), $i = 2,\, 3$. Choose two general points $R_4$ and 
$R_5$ of $C_0$ such that the line $L_1 \subset \piii$ joining them 
does not intersect $E_2$ and $E_3$ and is not contained in any of the surfaces 
$Q_2$ and $Q_3$. In this case $L_1 \cup E_i \cup C_0$ is contained in no 
quadric surface, $i = 2,\, 3$. One can choose, now, four general points 
$R_0 , \ldots , R_3$ of $C_0 \setminus (L_1 \cup E_2 \cup E_3)$ such that 
$L_1$, $E_2$, $E_3$ and $\{R_0 , \ldots , R_3\}$ satisfy the 
hypothesis of Lemma~\ref{L:3lines+4points} from Appendix~\ref{A:auxiliary} 
(see the proof of the case $g = 11$ of Theorem~\ref{T:main} in 
Section~\ref{S:g=11}). 
In this case, taking $W = \{R_0 , \ldots , R_5\}$ and applying 
Cor.~\ref{C:2lines+6points}, one gets that the multiplication map 
$\tH^0(\sci_{E_2 \cup E_3 \cup W}(3)) \otimes S_1 \ra 
\tH^0(\sci_{E_2 \cup E_3 \cup W}(4))$ is surjective.\hfill $\Box$       

\section{The Cases $5 \leq g \leq 7$}\label{S:g=5-7} 

In these cases, the least integer $d$ for which $\rho(g,3,d) \geq 0$ is 
$d = g + 2$. The method of Chang and Ran, as formulated in 
Section~\ref{S:method}, does not work anymore because, assuming that the curve 
$Y$ is linearly normal, one has $\h^0(\omega_Y(-1)) = \h^1(\sco_Y(1)) = 1$ and 
$\text{deg}\, \omega_Y(-1) = g - 4 > 0$ hence $\omega_Y(-1)$ cannot be 
globally generated (a condition that was necessary for the construction of a 
vector bundle). Things can be, however, fixed by working with rank 2 reflexive 
sheaves instead of vector bundles (this is, actually, the approach from 
Chang \cite[Cor.~4.8.1]{c}).   

More precisely, if $Y \subset \piii$ is a linearly normal, nonsigular, 
connected curve of genus $g$, $5 \leq g \leq 7$, and degree $d = g + 2$ then, 
as we saw above, $\h^0(\omega_Y(-1)) = 1$. This implies immediately that 
$\mu_0(Y)$ is injective. Moreover, for reasons of degree and genus, 
$\tH^0(\sci_Y(2)) = 0$. A non-zero global section of $\omega_Y(-1)$ defines an 
extension$\, :$ 
\[
0 \lra \sco_\piii \lra \sce \lra \sci_Y(5) \lra 0\, , 
\]     
with $\sce$ a rank 2 reflexive sheaf with $c_1(\sce) = 5$.  
Consider the ``normalized'' rank 2 reflexive sheaf $\scf := \sce(-3)$, with 
Chern classes $c_1(\scf) = -1$, $c_2(\scf) = \text{deg}\, Y - 6 = g - 4$, 
$c_3(\scf) = \text{deg}\, \omega_Y(-1) = g - 4$. One has $\tH^0(\scf) = 0$ 
(because $\tH^0(\sci_Y(2)) = 0$) and $\tH^1(\scf(-1)) = 0$ (because 
$\tH^1(\sci_Y(1)) = 0$). Now, one has$\, :$ 

\begin{lemma}\label{L:monadscf} 
Let $g$ be an integer with $5 \leq g \leq 7$ and let $\scf$ be a rank $2$ 
reflexive sheaf on $\piii$ with Chern classes 
$c_1(\scf) = -1$, $c_2(\scf) = g - 4$, $c_3(\scf) = g - 4$. If 
${\fam0 H}^0(\scf) = 0$ and ${\fam0 H}^1(\scf(-1)) = 0$ then $\scf(1)$ is the 
cohomology sheaf of a monad of the form$\, :$ 
\begin{equation}\label{E:monadgleq7} 
0 \lra (g - 4)\sco_\piii(-1) \overset{\beta}{\lra} (2g - 7)\sco_\piii 
\overset{\alpha}{\lra} (g - 5)\sco_\piii(1) \lra 0\, . 
\end{equation} 
Here  ``monad'' means that $\alpha$ is an epimorphism, $\beta^\vee$ is an  
epimorphism except at finitely many points, and $\alpha \circ \beta = 0$.  
\end{lemma} 

\begin{proof}  
We use the properties of the \emph{spectrum} of a stable rank 2 reflexive 
sheaf from Hartshorne \cite[Sect.~7]{ha}. Since $\tH^0(\scf) = 0$, $\scf$ 
is stable. Since $\tH^1(\scf(-1)) = 0$, the spectrum of $\scf$ must be 
$k_\scf = (-1, \ldots ,-1)$ ($g - 4$ times).  
It follows that $\tH^1(\scf(l)) = 0$ for $l \leq -1$ 
and $\tH^2(\scf(l)) = 0$ for $l \geq -1$. Since $\tH^2(\scf(-1)) = 0$ and 
$\tH^3(\scf(-2)) \simeq \tH^0(\scf^\vee(-2))^\vee \simeq \tH^0(\scf(-1))^\vee = 
0$, the Castelnuovo-Mumford lemma (in the form stated in  
\cite[Lemma~1.21]{acm1}) implies that the graded $S$-module $\tH^1_\ast(\scf)$ 
is generated in degrees $\leq 0$, hence it is generated by $\tH^1(\scf)$. 
By Riemann-Roch, $\h^1(\scf) = -\chi(\scf) = g - 5$. Consider, now, the 
universal extension$\, :$ 
\[
0 \lra \scf \lra \scg \lra (g - 5)\sco_\piii \lra 0\, . 
\] 
$\scg$ is a rank 4 reflexive sheaf with $\tH^1_\ast(\scg) = 0$ and with 
$\tH^0(\scg) = 0$. Since $\tH^2(\scg(-1)) \simeq \tH^2(\scf(-1)) = 0$ and 
$\tH^3(\scg(-2)) \simeq \tH^3(\scf(-2)) = 0$, $\scg$ is 1-regular. 
By Riemann-Roch$\, :$ 
\[
\h^0(\scg(1)) = \h^0(2\sco_\piii(1)) + \h^0(\scf(1)) - \h^1(\scf(1)) 
= \h^0(2\sco_\piii(1)) + \chi(\scf(1)) = 2g - 7\, . 
\] 
The kernel $K$ of the evaluation epimorphism $(2g - 7)\sco_\piii \ra \scg(1)$ 
of $\scg(1)$ is locally free of rank $g - 4$. One has $\tH^1_\ast(K) = 0$ 
and $\tH^2_\ast(K) \simeq \tH^1_\ast(\scg(1)) = 0$, hence $K$ is a direct sum of 
line bundles. Since $c_1(K) = -(g - 4)$ and $\tH^0(K) = 0$ it follows that 
$K \simeq (g - 4)\sco_\piii(-1)$. 
\end{proof}  

Conversely, if $\scf(1)$ is the cohomology sheaf of a monad of the form 
\eqref{E:monadgleq7} then $\scf$ is a rank 2 reflexive sheaf with 
$\tH^0(\scf) = 0$ and $\tH^1(\scf(-1)) = 0$. Moreover, since $g - 5 \leq 2$,  
Lemma~\ref{L:5ora2o(1)} from Appendix~\ref{A:monads} implies that  
$\tH^1(\scf(2)) = 0$ hence $\scf$ is 3-regular. 
In particular, $\scf(3)$ is globally generated. If, moreover, 
$\sce xt_{\sco_\piii}^1(\scf , \sco_\piii) \simeq \sco_\Gamma$, where $\Gamma$ is 
a 0-dimensional subscheme of $\piii$ consisting of simple points, then the 
zero scheme of a general global section of $\scf(3)$ is a \emph{nonsigular} 
curve $Y$ (see, for example, \cite[Prop.~4]{bc}). $Y$ has genus $g$ 
and degree $d = g + 2$, and is linearly normal. 

Finally, the monads of the form \eqref{E:monadgleq7} can be put together into 
a family with irreducible, rational base because, by Lemma~\ref{L:5ora2o(1)} 
from Appendix~\ref{A:monads}, $\tH^0(\alpha(1))$ is surjective, for any 
such monad (again, since $g - 5 \leq 2$). Consequently, in order to prove the 
unirationality of $\mathcal{M}_g$, for $5 \leq g \leq 7$, it suffices to show 
that there exist rank 2 reflexive sheaves $\scf$, with $c_1(\scf) = -1$, 
$c_2(\scf) = g - 4$, $c_3(\scf) = g - 4$, such that $\tH^0(\scf) = 0$, 
$\tH^1(\scf(-1)) = 0$ and such that $\sce xt_{\sco_\piii}^1(\scf , \sco_\piii)$ is 
as above. One can construct such sheaves as general extensions$\, :$
\[
0 \lra \sco_\piii(-2) \lra \scf \lra \sci_C(1) \lra 0\, , 
\]  
with $C$ a rational curve of degree $g - 2$. Here general means that the 
extension is defined by a global section of $\omega_C(1) \simeq 
\sco_\pj(g - 4)$ vanishing in $g - 4$ distinct points. 

\begin{remark}\label{R:g=4} 
In the case $g = 4$ one gets $d = 6$. 
Then $\omega_Y(-1)$ has degree 0 and, since $\h^0(\sco_Y(1)) - 
\h^0(\omega_Y(-1)) = 3$ and $\h^0(\sco_Y(1)) \geq 4$, $\h^0(\omega_Y(-1)) > 0$. 
It follows that $\omega_Y(-1) \simeq \sco_Y$. A non-zero global section of 
$\omega_Y(-1)$ defines an extension $0 \ra \sco_\piii(-3) \ra \scf \ra 
\sci_Y(2) \ra 0$, with $\scf$ locally free of rang 2, with $c_1(\scf) = -1$, 
$c_2(\scf) = 0$. Since $\tH^0(\scf) \simeq \tH^0(\sci_Y(2)) \neq 0$ (because 
$\tH^1(\sco_Y(2)) = 0$ hence $\h^0(\sco_Y(2)) = 9$) and 
$\tH^0(\scf(-1)) = 0$, one gets that $\scf \simeq \sco_\piii \oplus 
\sco_\piii(-1)$ hence $Y$ is a complete intersection of type $(2,3)$. One can 
view, perhaps, the whole story above as a generalization of this simple fact.   
\end{remark}

\appendix 
\section{Monads and Dependency Loci}\label{A:monads}

\begin{lemma}\label{L:quasilinearmonad} 
Let $F$ be a vector bundle on $\piii$. If ${\fam0 H}^0(F(-1)) = 0$, 
${\fam0 H}^1(F(-2)) = 0$, ${\fam0 H}^2(F(-1)) = 0$, ${\fam0 H}^2(F(-2)) = 0$ 
and ${\fam0 h}^3(F(-3)) \leq 3$ then $F$ is the cohomology sheaf of a 
monad of the form$\, :$ 
\[
0 \ra {\fam0 H}^1(F^\vee(-1))^\vee \otimes \sco(-1) \ra 
\begin{matrix}
a\sco\\ \oplus\\ {\fam0 H}^0(F^\vee(-1))^\vee \otimes \sco(-1) 
\end{matrix}
\ra {\fam0 H}^1(F(-1)) \otimes \sco(1) \ra 0.
\]
\end{lemma} 

\begin{proof} 
$\h^0(F^\vee(-1)) = \h^3(F(-3)) \leq 3$ implies that $\tH^0(F^\vee(-2)) = 0$ 
hence, by Serre duality, $\tH^3(F(-2)) = 0$. Now, the following assertions are 
consequences of the Castelnuovo-Mumford lemma (in the form stated in 
\cite[Lemma~1.21]{acm1})$\, :$ 

\vskip2mm 

\noindent 
$\bullet$\quad $\tH^2(F(-1)) = 0$ and $\tH^3(F(-2)) = 0$ imply that the graded 
$S$-module $\tH^1_\ast(F)$ is generated in degrees $\leq 0$. 

\noindent 
$\bullet$\quad $\h^2(F^\vee(-2)) = \h^1(F(-2)) = 0$ and $\h^3(F^\vee(-3)) = 
\h^0(F(-1)) = 0$ imply that $\tH^2(F^\vee(l)) = 0$ for $l \geq -2$ hence 
$\tH^1(F(l)) = 0$ for $l \leq -2$. 

\noindent 
$\bullet$\quad $\h^2(F^\vee(-2)) = \h^1(F(-2)) = 0$ and $\h^3(F^\vee(-3)) = 
\h^0(F(-1)) = 0$ imply that the graded $S$-module $\tH^1_\ast(F^\vee)$ is 
generated in degrees $\leq -1$. 

\noindent 
$\bullet$\quad $\tH^2(F(-1)) = 0$ and $\tH^3(F(-2)) = 0$ imply that 
$\tH^2(F(l)) = 0$ for $l \geq -1$ hence $\tH^1(F^\vee(l)) = 0$ for $l \leq -3$. 
Moreover, $\tH^1(F^\vee(-2)) \simeq \tH^2(F(-2))^\vee = 0$. 

\vskip2mm 

Consequently, $\tH^1_\ast(F)$ has $\h^1(F(-1))$ minimal generators in degree 
$-1$ and some number $b$ of minimal generators in degree 0, and  
$\tH^1_\ast(F^\vee)$ is generated by $\tH^1(F^\vee(-1))$. Now, applying Horrocks' 
method of ``killing cohomology'' (explained in Barth and Hulek \cite{bh}), one 
gets that $F$ is the cohomology sheaf of a monad of the form$\, :$ 
\[
0 \ra \tH^1(F^\vee(-1))^\vee \otimes_k \sco_\piii(-1) \overset{\beta}{\lra} A 
\overset{\alpha}{\lra} \tH^1(F(-1)) \otimes_k \sco_\piii(1) \oplus b\sco_\piii 
\ra 0\, ,   
\]
with $A$ is a direct sum of line bundles. One has $\tH^0(A(-1)) = 0$ (since  
$\tH^0(F(-1)) = 0$), $\h^0(A^\vee(-1)) = \h^0(F^\vee(-1))$ and $\tH^0(A^\vee(-2)) 
= 0$ (because $\tH^0(F^\vee(-2)) = 0$). One deduces that $A \simeq a\sco_\piii 
\oplus \tH^0(F^\vee(-1))^\vee \otimes \sco_\piii(-1)$, for some integer $a$. 
Since the component $a\sco_\piii \ra b\sco_\piii$ of $\alpha$ is 0 (because 
$\tH^1_\ast(F)$ has $b$ minimal generators of degree 0) and since there is no 
epimorphism $3\sco_\piii(-1) \ra \sco_\piii$ one deduces that $b = 0$.    
\end{proof}

\begin{corollary}\label{C:linearmonad} 
Let $F$ be a vector bundle on $\piii$. If ${\fam0 H}^0(F(-1)) = 0$, 
${\fam0 H}^1(F(-2)) = 0$, ${\fam0 H}^2(F(-2)) = 0$ and 
${\fam0 H}^3(F(-3)) = 0$ then $F$ is the cohomology sheaf of a linear 
monad of the form$\, :$ 
\[
0 \ra {\fam0 H}^1(F^\vee(-1))^\vee \otimes_k \sco_\piii(-1) \ra 
a\sco_\piii \ra 
{\fam0 H}^1(F(-1)) \otimes_k \sco_\piii(1) \ra 0\, . 
\]
Moreover, ${\fam0 h}^1(F(-1)) = - \chi(F(-1))$ and ${\fam0 h}^1(F^\vee(-1)) = 
- \chi(F^\vee(-1))$.  
\end{corollary} 

\begin{proof} 
Since $\tH^2(F(-2)) = 0$ and $\tH^3(F(-3)) = 0$, the Lemma of 
Castelnuovo-Mumford implies that $\tH^2(F(-1)) = 0$ (and $\tH^3(F(l)) = 0$ 
for $l \geq -3$). One can apply, now, Lemma~\ref{L:quasilinearmonad}. 
For the last assertions from the statement, one notices that $\tH^i(F(-1)) 
= 0$ for $i \neq 1$ and that $\tH^i(F^\vee(-1)) = 0$ for $i \neq 1$ (because 
$F^\vee$ satisfies the hypothesis of the corollary, too). 
\end{proof}

\begin{remark}\label{R:linearmonad} 
It is easy to see that, conversely, if a vector bundle $F$ is the cohomology 
sheaf of a linear monad $0 \ra a\sco_\piii(-1) \ra b\sco_\piii \ra 
c\sco_\piii(1) 
\ra 0$ then it satisfies the hypothesis of Cor.~\ref{C:linearmonad}. 
\end{remark} 

\begin{lemma}\label{L:h1evbh=0} 
Let $E$ be a vector bundle on $\p^n$, $n \geq 2$. If ${\fam0 H}^1(E_H) = 0$,  
for every hyperplane $H \subset \p^n$, then ${\fam0 h}^1(E) \leq 
{\fam0 max}(0 , {\fam0 h}^1(E(-1)) - n)$. 
\end{lemma}

\begin{proof}
If $H \subset \p^n$ is a hyperplane of equation $h = 0$ then, using the 
exact sequence$\, :$ 
\[
\tH^1(E(-1)) \overset{h}{\lra} \tH^1(E) \lra \tH^1(E_H) = 0, 
\]
one deduces that the multiplication by the linear form $h \colon \tH^1(E(-1)) 
\ra \tH^1(E)$ is surjective. Applying, now, the Bilinear Map Lemma 
\cite[Lemma~5.1]{ha}, to the bilinear map $\tH^1(E)^\vee \times 
\tH^0(\sco_{\p^n}(1)) \ra \tH^1(E(-1))^\vee$ deduced from the multiplication 
map $\tH^1(E(-1)) \times \tH^0(\sco_{\p^n}(1)) \ra \tH^1(E)$, one gets the 
inequality from the statement. 
\end{proof}

\begin{lemma}\label{L:5ora2o(1)} 
If $\phi \colon m\sco_{\p^n} \ra 2\sco_{\p^n}(1)$ is an epimorphism of vector 
bundles on $\p^n$, $n \geq 1$, then ${\fam0 H}^0(\phi(1)) \colon 
{\fam0 H}^0(m\sco_{\p^n}(1)) \ra {\fam0 H}^0(2\sco_{\p^n}(2))$ is surjective.  
\end{lemma}

\begin{proof} 
Since $\phi$ is an epimorphism, $\tH^0(\phi) \colon \tH^0(m\sco_{\p^n}) \ra 
\tH^0(2\sco_{\p^n}(1))$ must have rank $\geq n + 2$. The kernel $K$ of $\phi$ 
is a vector bundle on $\p^n$ with $\h^1(K) \leq 2(n + 1) - (n + 2) = n$. 
We shall prove, by induction on $n$, that $\tH^1(K(1)) = 0$. 

The case $n = 1$ is clear because, in that case, $K \simeq (m - 4)\sco_\pj 
\oplus 2\sco_\pj(-1)$ or $K \simeq (m - 3)\sco_\pj \oplus \sco_\pj(-2)$. 

Assuming that our assertion is true on $\p^{n-1}$ let us prove it on $\p^n$. 
By the induction hypothesis, $\tH^1(K_H(1)) = 0$, for every hyperplane 
$H \subset \p^n$. As we saw at the beginning of the proof, $\h^1(K) \leq n$. 
Applying Lemma~\ref{L:h1evbh=0} to $E := K(1)$ one gets that $\tH^1(K(1)) = 
0$.    
\end{proof} 

The next lemma is a weak variant of some results of Martin-Deschamps and 
Perrin \cite{mdp}. This variant suffices for our purposes. We include, for 
the reader's convenience, an argument that we have extracted from (several 
places of) the paper of Martin-Deschamps and Perrin.  

\begin{lemma}\label{L:dependencyloci} 
Let $E$ be a vector bundle on $\piii$, of rank $r \geq 2$, such that the 
evaluation morphism ${\fam0 ev}_E \colon {\fam0 H}^0(E) \otimes_k \sco_\piii 
\ra E$ has rank $\geq r - 1$ at every point of $\piii$. Let $\Delta$ be the 
degeneracy scheme of ${\fam0 ev}_E$. If $\dim \Delta \leq 1$ and if there are 
only finitely many poins $x \in \Delta$ for which $\sci_{\Delta , x} \subseteq 
\fm_x^2$ (where $\fm_x$ is the maximal ideal of the local ring 
$\sco_{\piii , x}$) then the dependency scheme of $r - 1$ general global 
sections of $E$ is a nonsingular (but not necessarily connected) curve.   
\end{lemma} 

\begin{proof} 
Let us, firstly, recall a definition$\, :$ let $\phi \colon F \ra E$ be a 
morphism of vector bundles on $\piii$ (or on any scheme) and $i \geq 0$  
an integer. One denotes by $D_i(\phi)$ the zero scheme of 
$\overset{i+1}{\bigwedge} \phi$, viewed as a global section of 
$\sch om(\overset{i+1}{\bigwedge} F , \overset{i+1}{\bigwedge} E) \simeq 
(\overset{i+1}{\bigwedge} F)^\vee \otimes 
\overset{i+1}{\bigwedge} E$. If $\text{rk}\, F \geq \text{rk}\, E =: r$ then 
the \emph{degeneracy scheme} of $\phi$ is, by definition, $D_{r-1}(\phi)$. 
Let $\phi^\prime \colon F^\prim \ra E^\prim$ be another morphism of vector 
bundles. If $\Cok \phi^\prime$ is isomorphic, locally on $\piii$, to 
$\Cok \phi$ then $D_{r^\prim - i}(\phi^\prime) = D_{r-i}(\phi)$, $\forall \, i 
\geq 1$, $r^\prim$ being the rank of $E^\prim$$\, ;$ this follows from the basic 
property of Fitting ideals (see, for example, Eisenbud \cite[\S~20.2]{eis}). 

Next, if $E$ is a rank $r$ vector bundle on $\piii$ let $\p(\tH^0(E))$ denote 
the (classical) projective space of 1-dimensional $k$-vector subspaces of 
$\tH^0(E)$. Assuming that $\h^0(E) = N + 1$, one has $\p(\tH^0(E)) \simeq 
\p^N$. Consider the following closed subscheme of $\p(\tH^0(E)) \times 
\piii$$\, :$ 
\[
Z := \{ ([s] , x) \vb s(x) = 0 \}\, , 
\]    
and the canonical projections $p \colon Z \ra \p(\tH^0(E))$ and $q \colon Z 
\ra \piii$. The fiber $p^{-1}([s])$ can be identifed with the zero scheme of 
the global section $s$ of $E$, while $q$ turns $Z$ into a $\p^{N-i}$-bundle 
over (the scheme) $D_i(\text{ev}_E) \setminus D_{i-1}(\text{ev}_E)$. In 
particular, the singular locus of $Z$ is contained in 
$q^{-1}(D_{r-1}(\text{ev}_E))$. We also notice that if one has an exact 
sequence$\, :$ 
\[
0 \lra m\sco_\piii \lra E \lra E^\prim \lra 0\, , 
\]
with $E^\prim$ a vector bundle of rank $r - m$ then $D_i(\text{ev}_{E^\prim}) = 
D_{i+m}(\text{ev}_E)$, $\forall \, i \geq 0$, because $\Cok \text{ev}_{E^\prim} 
\simeq \Cok \text{ev}_E$. 

It follows, now, easily, by decreasing induction on $r \geq 3$, that if 
$\text{ev}_E$ has rank $\geq r - 2$ at every point of $\piii$, if 
$\dim D_{r-1}(\text{ev}_E) \leq 1$ and if $\dim D_{r-2}(\text{ev}_E) \leq 0$ 
then the dependency scheme $\Gamma$ of $r - 2$ general global sections of 
$E$ consists of finitely many simple points. In this case, one has an exact 
sequence$\, :$ 
\[
0 \lra (r - 2)\sco_\piii \lra E \lra \scf \lra 0\, , 
\]  
where $\scf$ is a rank 2 reflexive sheaf with 
$\sce xt^1_{\sco_\piii}(\scf , \sco_\piii) \simeq \sco_\Gamma$. Let $\sigma$ denote 
the restriction of the evaluation morphism $\tH^0(\scf) \otimes_k \sco_\piii 
\ra \scf$ of $\scf$ to $\piii \setminus \Gamma$, let $W$ be the closed 
subscheme of $\p(\tH^0(\scf)) \times (\piii \setminus \Gamma)$ analogous to 
$Z$ above, and let $\pi \colon W \ra \p(\tH^0(\scf))$ and $\rho \colon W \ra 
\piii \setminus \Gamma$ be the canonical projections. Notice that 
$\tH^0(\scf)$ has dimension $N^\prime + 1$, where $N^\prime = N - r + 2$.  

Under the hypothesis of the lemma, $D_1(\sigma) = \Delta \setminus \Gamma$ and 
$D_0(\sigma) = \emptyset$. $W$ is given, locally, by two equations. Since it 
is a $\p^{N^\prime - 2}$-bundle over $\piii \setminus (\Delta \cup \Gamma)$ and a 
$\p^{N^\prime - 1}$-bundle over $\Delta \setminus \Gamma$, one deduces that $W$ is 
irreducible of dimension $N^\prime + 1$ and $\text{Sing}\, W \subseteq 
\rho^{-1}(\Delta \setminus \Gamma)$. 

$\Delta \setminus \Gamma$ can be covered with open subsets $U$ of $\piii 
\setminus \Gamma$ with the property that $\scf \vb U$ is trivial and there 
exists a global section $s_0$ of $\scf$ vanishing at no point of $U$. Extend 
$s_0 \vb U$ to a local frame $(s_0 \vb U\, ,\, t)$ of $\scf \vb U$ and $s_0$ to 
a $k$-basis  $s_0 , \ldots , s_{N^\prime}$ of $\tH^0(\scf)$. Then$\, :$ 
\[
s_i \vb U = f_i(s_0 \vb U) + g_it\, ,\  \text{with}\  f_i,\, g_i \in 
\sco_\piii(U)\, ,\  i = 1, \ldots , N^\prime\, . 
\]
One can assume that $U$ is isomorphic to an open subset of the affine space 
$\mathbb{A}^3$ hence $f_i$ and $g_i$ are functions in three variables 
$x_1,\, x_2,\, x_3$. $W \cap (\p(\tH^0(\scf)) \times U)$ is given by the 
equations$\, :$ 
\[
\lambda_0 + {\textstyle \sum}_{i = 1}^{N^\prime} \lambda_if_i(x) = 0\, ,\  
{\textstyle \sum}_{i = 1}^{N^\prime} \lambda_ig_i(x) = 0\, , 
\]
$\lambda_0, \ldots , \lambda_{N^\prime}$ being homogeneous coordinates on 
$\p(\tH^0(\scf)) \simeq \p^{N^\prime}$. The Jacobian matrix of this system of 
equations is$\, :$ 
\[ 
\begin{pmatrix} 
1 & f_1(x) & \cdots & f_{N^\prime}(x) & 
{\textstyle \sum}\lambda_i(\partial f_i/\partial x_1)(x) & 
{\textstyle \sum}\lambda_i(\partial f_i/\partial x_2)(x) & 
{\textstyle \sum}\lambda_i(\partial f_i/\partial x_3)(x)\\ 
0 & g_1(x) & \cdots & g_{N^\prime}(x) & 
{\textstyle \sum}\lambda_i(\partial g_i/\partial x_1)(x) & 
{\textstyle \sum}\lambda_i(\partial g_i/\partial x_2)(x) & 
{\textstyle \sum}\lambda_i(\partial g_i/\partial x_3)(x)
\end{pmatrix} . 
\]
One deduces that if $x \in U$ then $\rho^{-1}(x) \subseteq \text{Sing}\, W$ 
if and only if$\, :$ 
\[
g_i(x) = 0\, ,\, (\partial g_i/\partial x_1)(x) = 0 \, ,\, 
(\partial g_i/\partial x_2)(x) = 0 \, ,\, 
(\partial g_i/\partial x_3)(x) = 0 \, ,\  i = 1 , \ldots , N^\prime \, ,  
\]
and this is equivalent to $(g_i)_x \in \fm_x^2$, $i = 1 , \ldots , N^\prime$. 
Since the ideal $\sci_{\Delta , x}$ of $\sco_{\piii , x}$ is generated by 
$(g_i)_x$, $i = 1 , \ldots , N^\prime$, it follows that there are only finitely 
many points $x$ of $\Delta \setminus \Gamma$ for which the fiber 
$\rho^{-1}(x)$ is entirely contained in $\text{Sing}\, W$. This implies that 
$\dim \text{Sing}\, W \leq N^{\prime} - 1$ hence $\pi(\text{Sing}\, W)$ is 
not dense in $\p(\tH^0(\scf))$. Applying the Theorem of generic smoothness, 
one gets that the zero scheme of a general global section $s$ of $\scf$ is 
a curve whose singular locus is contained in $\Gamma$. (The above argument 
appears in the proof of \cite[IV,~Prop.~3.1]{mdp}.) 

On the other hand, if $s \in \tH^0(\scf)$ and $x \in \Gamma$ then the zero 
scheme of $s$ contains $x$ and it is, locally at $x$, a nonsingular curve if 
and only if $s(x) \neq 0$ in $\scf(x) := \scf_x/\fm_x\scf_x$ (see, for 
example, the proof of \cite[Prop.~3]{bc}). Since the evaluation map 
$\tH^0(\scf) \ra \scf(x)$ is non-zero (it has, in fact, rank $\geq 2$) the 
lemma is proven.    
\end{proof}

\begin{remark}\label{R:dependencyloci} 
(i) The hypothesis of Lemma~\ref{L:dependencyloci} is verified if 
there exists a closed subscheme $\Delta$ of $\piii$, of dimension 
$\leq 1$, containing only finitely many points $x$ for which 
$\sci_{\Delta , x} \subseteq \fm_x^2$, such that the cokernel of the evaluation 
morphism of $E$ is an invertible $\sco_\Delta$-module. 

(ii) $\sci_{\Delta , x}$ is not contained in $\fm_x^2$ if and only if there 
exists an open neighbourhood $U$ of $x$ in $\piii$ such that $U \cap \Delta$ 
is contained in a nonsingular surface.   
\end{remark}

\section{Lines and Points in $\piii$}\label{A:auxiliary}

\begin{lemma}\label{L:3lines+4points} 
Let $Y$ be the union of three mutually disjoint lines $L_1,\, L_2,\, L_3$ in 
$\piii$ and let $P_0 , \ldots , P_3$ be four points in $\piii$ such that 
$W := \{P_0 , \ldots , P_3\}$ is not contained in a plane. We assume that 
none of the four points belongs to the quadric surface $Q \subset \piii$ 
containing $Y$ and that ${\fam0 H}^0(\sci_{(Y \setminus L_l) \cup W}(2)) 
= 0$, $l = 1,\, 2,\, 3$. Then the homogeneous ideal of $Y \cup W$ is 
generated by cubic forms.  
\end{lemma} 

\begin{proof} 
Since $\tH^0(\sci_{(Y \setminus L_l) \cup W}(2)) = 0$ it follows that 
$\tH^0(\sci_{Y \setminus L_l}(2)) \izo \tH^0(\sco_W(2))$, $l = 1,\, 2,\, 3$. 
One deduces easily that the map $\tH^0(\sci_Y(3)) \ra \tH^0(\sco_W(3))$ 
is surjective hence $\tH^1(\sci_{Y \cup W}(3)) = 0$ and 
$\h^0(\sci_{Y \cup W}(3)) = 4$. Consequently, $\sci_{Y \cup W}$ is 
4-regular. It remains to show that the multiplication map  
$\mu \colon \tH^0(\sci_{Y \cup W}(3)) \otimes S_1 \ra \tH^0(\sci_{Y \cup W}(4))$  
is surjective. Using the commutative diagram$\, :$ 
\[
\SelectTips{cm}{12}\xymatrix{
0\ar[r] & \tH^0(\sci_{Y \cup W}(3)) \otimes S_1\ar[r]\ar[d]^{\mu} &  
\tH^0(\sci_Y(3)) \otimes S_1\ar[r]\ar[d]^{\mu_Y} &  \tH^0(\sco_W(3)) \otimes 
S_1\ar[r]\ar[d]^{\mu_W} & 0\\ 
0\ar[r] & \tH^0(\sci_{Y \cup W}(4))\ar[r] & \tH^0(\sci_Y(4))\ar[r] & 
\tH^0(\sco_W(4))\ar[r] & 0}
\] 
one sees that it suffices to show that the map $\Ker \mu_Y \ra \Ker \mu_W$ 
induced by this diagram is surjective. 

For $0 \leq i \leq 3$, let $h_i = 0$ be an equation of the plane containing 
$W \setminus \{P_i\}$ and let $e_i$ be the element of $\tH^0(\sco_W)$ 
defined by $e_i(P_j) = \delta_{ij}$, $j = 0, \ldots , 3$. Since 
$\mu_W(e_i \otimes h_j) = h_j(P_i)e_i$ it follows that $\Ker \mu_W$ has 
a $k$-basis consisting of the elements $e_i \otimes h_j$, with $0 \leq i \leq 
3$, $0 \leq j \leq 3$ and $i \neq j$. 

Take an $i \in \{0, \ldots , 3\}$. For $1 \leq l \leq 3$, let $h_{il} = 0$ 
be an equation of the plane containing $L_l \cup \{P_i\}$. We assert that 
$h_{i1},\, h_{i2},\, h_{i3}$ are linearly independent. \emph{Indeed}, if they 
are linearly dependent then they vanish on a line $L \subset \piii$ 
containing $P_i$. $L$ is, then, a 3-secant of $Y = L_1 \cup L_2 \cup L_3$ 
hence $L$ is contained in the quadric surface $Q \subset \piii$ containing 
$Y$ and this \emph{contradicts} the fact that $P_i \notin Q$. It, thus, 
remains that $h_{i1},\, h_{i2},\, h_{i3}$ are linearly independent. 

For $1 \leq l \leq 3$, let $q_{il} = 0$ be an equation of the unique quadric 
surface containing $(Y \setminus L_l) \cup (W \setminus \{P_i\})$ (recall 
that $\tH^0(\sci_{Y \setminus L_l}(2)) \izo \tH^0(\sco_W(2))$). Choose, also, 
$h_{il}^\prime \in S_1$ vanishing on $L_l$ but not at $P_i$. Then$\, :$ 
\[
h_{il}^\prime q_{il} \otimes h_{il} - h_{il}q_{il} \otimes h_{il}^\prime  
\]  
belongs to $\Ker \mu_Y$ and its image into $\Ker \mu_W$ is 
$(h_{il}^\prime q_{il})(P_i)e_i \otimes h_{il}$. Since $h_{il}^\prime$ and $q_{il}$ 
do not vanish at $P_i$, one deduces that $e_i \otimes h_{il}$ belongs to the 
image of $\Ker \mu_Y \ra \Ker \mu_W$. 

Finally, if $j \in \{0, \ldots , 3\} \setminus \{i\}$ then $h_j(P_i) = 0$. 
Since $h_{i1},\, h_{i2},\, h_{i3}$ vanish at $P_i$ and are linearly independent, 
it follows that $h_j$ is a linear combination of $h_{i1},\, h_{i2},\, h_{i3}$ 
hence $e_i \otimes h_j$ belongs to the image of $\Ker \mu_Y \ra 
\Ker \mu_W$. Since $i \in \{0, \ldots , 3\}$ and $j \in \{0, \ldots , 3\} 
\setminus \{i\}$ were arbitrary, the map $\Ker \mu_Y \ra \Ker \mu_W$ is 
surjective. 
\end{proof} 

\begin{corollary}\label{C:2lines+6points} 
Under the hypothesis of Lemma~\ref{L:3lines+4points}, choose two more points 
$P_4$ and $P_5$ on the line $L_1$ and put $W^\prim := \{P_0 , \ldots , P_5\}$. 
Then the homogeneous ideal of $L_2 \cup L_3 \cup W^\prim$ is generated by 
cubic forms.  
\end{corollary}

\begin{proof} 
One has an exact sequence$\, :$ 
\[
0 \lra \sci_{Y \cup W} \lra \sci_{L_2 \cup L_3 \cup W^\prim} \lra 
\sci_{\{P_4 , P_5\} , L_1} \lra 0 
\]
and $\sci_{\{P_4 , P_5\} , L_1} \simeq \sco_{L_1}(-2)$. Since, by the proof of 
Lemma~\ref{L:3lines+4points}, $\tH^1(\sci_{Y \cup W}(3)) = 0$ and the 
multiplication map $\tH^0(\sci_{Y \cup W}(3)) \otimes S_1 \ra 
\tH^0(\sci_{Y \cup W}(4))$ is surjective it follows that 
$\tH^1(\sci_{L_2 \cup L_3 \cup W^\prim}(3)) = 0$ and the multiplication map 
$\tH^0(\sci_{L_2 \cup L_3 \cup W^\prim}(3)) \otimes S_1 \ra 
\tH^0(\sci_{L_2 \cup L_3 \cup W^\prim}(4))$ is surjective.  
\end{proof}

\section*{Acknowledgements}



N. Manolache expresses his thanks to the 
Institute of Mathematics, Oldenburg University, especially to Udo Vetter, for 
warm hospitality during the preparation of this work.

\end{document}